\documentclass[12pt, a4paper]{article}
\usepackage{geometry}
\usepackage{titlesec} 
\usepackage{amsfonts}
\usepackage{amsthm}
\usepackage{amssymb}
\usepackage{amsmath}
\usepackage{theoremref}
\usepackage{enumerate}
\usepackage{bbm}
\usepackage{bm}
\usepackage{authblk}
\usepackage{mathtools}
\usepackage{a4wide}
\usepackage{xcolor}

\DeclarePairedDelimiter\abs{\lvert}{\rvert}%
\DeclarePairedDelimiter\norm{\lVert}{\rVert}%

\makeatletter
\let\oldabs\abs
\def\abs{\@ifstar{\oldabs}{\oldabs*}}
\let\oldnorm\norm
\def\norm{\@ifstar{\oldnorm}{\oldnorm*}}
\makeatother

\makeatletter
\g@addto@macro\bfseries{\boldmath}
\makeatother

\newcommand{\A}{\mathcal{A}}

\newcommand{\K}{\mathcal{K}}

\newcommand{\T}{\mathbb{T}}
\newcommand{\Ca}{\mathcal{C}}
\newcommand{\z}{\zeta}
\newcommand{\conj}[1]{\overline{#1}}
\newcommand{\D}{\mathbb{D}}

\newcommand{\Po}{\mathcal{P}}
\newcommand{\cD}{\conj{\mathbb{D}}}
\newcommand{\ip}[2]{\big\langle #1, #2 \big\rangle}
\newcommand{\dist}[2]{\text{dist}( #1, #2 ) }

\newcommand{\m}{\textit{m}}
\newcommand{\hil}{\mathcal{H}} 

\newcommand{\Hb}{\mathcal{H}(b)}
\newcommand{\hb}{\mathcal{H}(b)}

\renewcommand\Re{\operatorname{Re}}

\newtheorem{thm}{Theorem}[section]
\newtheorem{lemma}[thm]{Lemma}

\newtheorem{cor}[thm]{Corollary}
\newtheorem{prop}[thm]{Proposition}
\theoremstyle{definition}

\theoremstyle{definition}
\newtheorem{definition}{Definition}[section]

\newcommand{\Addresses}{{% additional braces for segregating \footnotesize
		\bigskip
		\footnotesize
		
		Adem Limani, \\ \textsc{Centre for Mathematical Sciences, Lund University, \\
			Lund, Sweden}\\
		\texttt{adem.limani@math.lu.se}
		
		\medskip
		
		Bartosz Malman, \\ \textsc{KTH Royal Institute of Technology, \\
			Stockholm, Sweden}\\
			\texttt{malman@kth.se}		
	}}

\begin{document}
\title{\textbf{Constructions of some families of smooth Cauchy transforms}}
\author{Adem Limani and Bartosz Malman}
\date{ }
\maketitle

\begin{abstract}
For a given Beurling-Carleson subset $E$ of the unit circle $\T$ which has positive Lebesgue measure, we give explicit formulas for measurable functions supported on $E$ such that their Cauchy transforms have smooth extensions from $\D$ to $\T$. The existence of such functions has been previously established by Khrushchev in 1978, in non-constructive ways by the use of duality arguments. We construct several particular families of such Cauchy transforms with a few applications in operator and function theory in mind. In one application, we give a new proof of irreducibility of the shift operator on certain Hilbert spaces of functions. In another application, we establish a permanence principle for inner factors under convergence in certain topologies. The applications lead to a self-contained duality proof of the density of smooth functions in a very large class of de Branges-Rovnyak spaces. This extends the previously known approximation results.
\end{abstract}

\section{Introduction}

\label{introsec}

Let $E$ be a closed subset of the unit circle $\T = \{ z \in \mathbb{C} : |z| = 1 \}$ of the complex plane $\mathbb{C}$, and let the notation $dm$ stand for the Lebesgue measure, normalized by the condition $m(\T) = 1$. The starting point of our development is the following question which has been studied and answered by Khrushchev in \cite{khrushchev1978problem}. Namely, what conditions on the set $E$ guarantee the existence of a non-zero measurable function $h$ supported on $E$ for which the \textit{Cauchy transform}, or \textit{Cauchy integral}, \begin{equation}
\label{cauchytransform} C_{h1_E}(z) := \int_{\T} \frac{h(\zeta)1_E(\zeta)}{1-z\conj{\zeta}} dm(\zeta) = \int_E \frac{h(\zeta)}{1-z\conj{\zeta}} dm(\zeta), \quad z \in \D, \end{equation} which is an analytic function in the unit disk $\D = \{ z \in \mathbb{C} : |z| < 1 \}$, can be extended to a continuous function on the closed disk $\cD$? What conditions on $E$ are necessary to assure existence of such a measurable function $h$ for which also the complex derivative $C_{h1_E}'$ admits such an extension? In the above formula, and throughout the article, $1_E$ denotes the indicator function of the set $E$.

For the question to be interesting, we should think of the set $E$ as being rather rough. Indeed, if $E$ contains an arc $A$, then certainly any smooth function $s :\T \to \mathbb{C}$ with support on $A$ will be transformed into a function $C_s$ which is a member of $\A^\infty$. Here, $\A^\infty$ denotes the algebra of analytic functions in $\D$ for which the derivatives of any order extend continuously to $\cD$. The containment $C_s \in \A^\infty$ follows in this case readily from the rapid rate of decay of Fourier coefficients $\{s_n\}_n$ of the smooth function $s$, and the fact that $C_s(z) = \sum_{n=0}^\infty s_nz^n$. 

By $\A$ we will denote the class of analytic functions in $\D$ which admit a continuous extension to $\cD$, and by $\A^n$ we denote those functions for which the $n$:th derivative admits such an extension, that is $f^{(n)} \in \A$. Thus $\A^\infty = \cap_{n \geq 1} \A^n$. Khrushchev in \cite{khrushchev1978problem} has solved the existence part of the above stated problem in full. For a general closed set $E$, he establishes the existence of a non-zero measurable function $h$, with support only on $E$, such that $C_h$ given by \eqref{cauchytransform} is in the class $\A$. Moreover, he proves that there exists a non-zero measurable function $h$ supported on $E$ for which the transform \eqref{cauchytransform} is a function in $\A^\infty$ essentially if and only if $E$ contains a \textit{Beurling-Carleson set} of positive Lebesgue measure. A set $E$ is a Beurling-Carleson set if it is closed and if the condition \begin{equation}
\label{BCdef} \sum_{n=1} |A_n| \log(1/|A_n|) < \infty
\end{equation} is satisfied, where $\{A_n\}_n$ is the  system of disjoint open subarcs of $\T$ union of which equals the complement $\T \setminus E$, and $|A|$ denotes the length of the arc $A$. The class of Beurling-Carleson sets has a rich history, and appears notably in the solution of boundary zero set problems for smooth analytic functions, and zero set problems for Bergman spaces (see Carleson's paper \cite{carlesonuniqueness} and Korenblum's paper \cite{korenblum1975extension}, for instance). In \cite{DBRpapperAdem}, the present authors found that Beurling-Carleson sets play an important role in smooth approximation theory in de Branges-Rovnyak spaces, another classical and well-studied family of Hilbert spaces of analytic functions. 

A notable feature of the proofs of the above mentioned results of Khrushchev in \cite{khrushchev1978problem} is that they are non-constructive. The existence of the measurable function $h$ is established by duality arguments, and an explicit formula for $h$ is lacking. A similar but slightly simplified duality proof appears also in \cite{havinbook}. In the case when $E$ is a general closed subset of $\T$, Khrushchev proves the existence of $h$ by a duality argument involving the classical theorem of Khintchine–Ostrowski which deals with simultaneous convergence of Nevanlinna class functions on $\D$ and a subset of $\T$ (see \cite{havinbook} or \cite{khrushchev1978problem} for a precise statement). In the case when $E$ contains a Beurling-Carleson set of positive measure, he first proves a variant of the Khintchine-Ostrowski theorem for a certain different class of functions, and asserts the existence of $h$ by a duality argument similar to the one in the first case. The mentioned variant of Khintchine-Ostrowski theorem has also been independently established by Kegejan in \cite{kegejanex}. 

One of the aims of this article is to show that, in the second case in which Beurling-Carleson sets and the class $\A^\infty$ is involved, the theorem of Khrushchev can be obtained in a rather elementary and explicit way by using modifications of other known constructions. Thus, in one of the main results, \thref{constrKhru}, we will give explicit formulas for measurable functions $h$ supported on a given Beurling-Carleson set $E$ for which $C_{h1_E} = C_h$ defined by formula \eqref{cauchytransform} is a function in $\A^\infty$. 

The article is part of a larger project of the authors which deals with connections between some aspects of the uncertainty principle in harmonic analysis, removal of singularities of Cauchy transforms, theory of subnormal operators, and approximation theory in de Branges-Rovnyak spaces. Some of these connections are described in the articles \cite{ptmuinnner} and \cite{DBRpapperAdem}. With this in mind, we specialize the construction in Sections \ref{uncertaintysec} and \ref{permanencesec} and use it in two principal applications. The results established in these sections are improvements of previous work in \cite{ptmuinnner}, in some cases with completely new proofs. They are then combined and further applied in Section \ref{dbrsec} in order to sharpen a non-constructive approximation result in the theory of de Branges-Rovnyak spaces which is found in \cite{DBRpapperAdem}. 

% A previous version of this article that was posted on ArXiv included also some examples of constructive smooth approximation procedures in a smaller class of de Branges-Rovnyak spaces. These results will be published elsewhere.

We shall now describe the studied applications in more detail. Section \ref{uncertaintysec} deals with the following general problem. Let $X$ be some space of functions defined on a domain in the complex plane which contains the constants, and is invariant under the \textit{forward shift operator} $M_z : f(z) \mapsto zf(z)$, where $z$ is the coordinate function (or identity function) of the complex plane. The assumptions imply that the analytic polynomials can be considered as a subset of $X$, and we consider their closure in $X$, which we denote by $D$. In many important cases the functions in $X$ live on the closed unit disk $\cD$, the operator $M_z$ is a contraction (in the sense that $\|M_zf\|_X \leq \|f\|_X$ is satisfied) and a question or assumption which appears in several contexts (see \cite{dbrcont}, \cite{kriete1990mean} and \cite{aleman2009nontangential}, for instance) is related to existence of invariant subspaces of the operator $M_z: D \to D$ on which it acts as an isometry. In the particular case $X = L^2(\mu)$, where $\mu$ is a positive Borel measure compactly supported in the complex plane, the closure of analytic polynomials is usually denoted by $\Po^2(\mu)$. These spaces, together with the shift operator $M_z$, provide a model for the class of subnormal operator (here \cite{conway1991theory} is a great reference).
  If $\mu$ is a positive measure of the form \begin{equation} \label{mueq}  d\mu = dA + 1_E d\m \end{equation} ($dA$ and $d\m$ being the area measure of $\D$ and Lebesgue measure of $\T$, respectively) then the condition that $M_z$ is completely non-isometric on the closure of polynomials $D := \Po^2(\mu)$ is precisely the condition which ensures that $\Po^2(\mu)$ can be identified with a genuine space of analytic functions in $\D$. If not, then $\Po^2(\mu)$ will contain as a subset a space of the form $L^2(1_Fd\m)$, for some measurable subset $F$ of $E$, on which $M_z$ obviously acts as an isometry. In the context of $\Po^2(\mu)$-spaces, the lack of a subspaces of the type $L^2(1_Fd\m)$ goes under the name of \textit{irreducibility} (see \cite{thomson1991approximation} and \cite{aleman2009nontangential}). It is known that if $E$ is a Beurling-Carleson set, then the corresponding shift operator will be completely non-isometric on $\Po^2(\mu)$. This follows essentially from Khrushchev's work in  \cite{khrushchev1978problem}. However, replacing $dA$ by a weighted version $w\,dA$ in \eqref{mueq}, where $w$ is some function which decays rapidly to zero near the boundary of $\D$, or replacing $E$ by a set more complicated than a Beurling-Carleson set, then it might very well happen that $M_z$ admits an invariant subspace on which it acts as an isometry (see \cite{kriete1990mean} and in particular \cite{khrushchev1978problem} for details). In Section \ref{uncertaintysec} we construct a special family of smooth Cauchy transforms and employ it in a functional analytic argument to establish that $M_z$ is completely non-isometric on a wide range of Hilbert spaces of analytic functions which are structurally similar to the $\Po^2(\mu)$-spaces induced by measures similar to \eqref{mueq}. This development in particular implies the above mentioned results for $\Po^2(\mu)$-spaces, and even their extensions from \cite{ptmuinnner}, but the method of proof is completely different, arguably much more straight-forward, and the result actually reaches further. This extension is important for our principal application to de Branges-Rovnyak spaces in Section \ref{dbrsec}. We remark that a wealth of information on the behaviour of $\Po^2(\mu)$-spaces which are spaces of analytic functions can be found in \cite{aleman2009nontangential}. 

In Section \ref{permanencesec}, an explicitly constructed family of smooth Cauchy transforms will be used to solve a problem of spectral theory of functions in the unit disk, which again extends results of \cite{ptmuinnner} to a wider range of spaces. In the problem setting, we fix a topological space of analytic functions in $\D$ which contains at least the algebra $H^\infty$ of bounded analytic functions in $\D$. Every bounded analytic function $f$ admits an inner-outer factorization into $f = \theta U = BS_\nu U$, where $B$ is the usual Blaschke product, $S_\nu$ is a singular inner function constructed from a positive singular measure $\nu$, and $U$ is an outer function (see \cite{garnett} for precise definitions). The factor $\theta = BS_\nu$ is the \textit{inner factor} of $f$. The factorization can be thought of as a type of spectral decomposition of the function $f$, and the measure $\nu$ is then the singular part of this decomposition. Now, if a sequence of bounded analytic functions $S_\nu f_n$ converges to a bounded analytic function $f$ in a certain topology, then one might ask what part of the measure $\nu$ must necessarily appear in the spectral decomposition of $f$ as a consequence of this convergence. This problem is intimately connected to approximation problems in model spaces and de Branges-Rovnyak spaces (see \cite{DBRpapperAdem} and \cite{smoothdensektheta}). A carefully constructed family of smooth Cauchy transforms will help us to implement a functional analytic argument and establish this \textit{inner factor permanence} result for a very large class of singular inner functions and a range of spaces which is larger than the one appearing in the article \cite{ptmuinnner}, in which the present authors investigated this principle for the topologies induced by the above mentioned $\Po^2(\mu)$-spaces.

The improvements of previous results mentioned in the last two paragraphs should be of independent interests, but they are inspired by our ultimate application to the approximation theory in de Branges-Rovnyak spaces $\hb$, which we present in the final Section \ref{dbrsec}. For background on the theory of $\hb$-spaces, see \cite{sarasonbook}, \cite{hbspaces1fricainmashreghi} and \cite{hbspaces2fricainmashreghi}. A basic problem in the theory is to identify what functions are contained in the space $\hb$ and how this depends on the structure of the symbol $b$, which is any analytic function mapping the disk $\D$ into itself. It has been established by Sarason (see \cite{sarasonbook}) that the analytic polynomials are contained and norm-dense in the space $\hb$ if and only if the weight \[\Delta(\zeta) := 1-|b(\zeta)|^2, \quad \zeta \in \T\] has an integrable logarithm on $\T$: \begin{equation}
    \int_\T \log (\Delta) d\m > -\infty.
\end{equation} Moreover, it is also known that any $\hb$ space contains a dense subset of functions in $\A$ (see \cite{comptesrenduscont}). In particular, the inner factor $\theta$ of $b$ plays no role in the context of approximations by analytic polynomials or functions continuous up to the boundary. The situation is different in the context of approximations by functions in the class $\A^\infty$ or $\A^n$, or even the Hölder classes. It is known from \cite{DBRpapperAdem} that the smooth approximation problem is intimately connected with the two problems mentioned above: irreducibility problem for certain $\Po^2(\mu)$-spaces, and the inner factor permanence problem. Consequently, a combination of results in \cite{DBRpapperAdem} and \cite{ptmuinnner} shows that the functions in the class $\A^n$ will be dense in the space $\hb$ if the outer factor of $b$ is "good" and the "bad" part of the singularities of the inner factor of $b$ is appropriately located on $\T$, with respect to the outer factor. More precisely, it was found in \cite{dbrcont} that if weight $\Delta$ appearing above is of the form \begin{equation}
    \label{deltastruct} 
    \Delta = \sum_{n=1}^\infty w_n 1_{E_n},
\end{equation} where each set $E_n$ is a Beurling-Carleson set of positive Lebesgue measure, and $w_n$ is a non-negative weight satisfying \begin{equation} \label{deltalogint}
    \int_{E_n} \log (w_n) d\m > -\infty,
\end{equation} then the functions in $\A^n$ are dense in $\hb$ if $b$ is outer. Note that the two conditions above say something about the "good" structure of the support set of the weight $\Delta$, and something about $\Delta$ not being too small on the support. In \cite{dbrcont} examples are highlighted in which bad support and small size of $\Delta$ both independently prohibit such approximations in $\hb$, not only by functions in $\A^\infty$ but even by functions in the Hölder classes. In the presence of a non-trivial inner factor $\theta = BS_{\nu}$ of $b$, results of \cite{DBRpapperAdem} and \cite{ptmuinnner} show that what matters is the location on $\T$ of the support of a certain part of the singular measure $\nu$. To describe this mechanism, we will need to introduce a simple decomposition of the measure $\nu$ which has appeared already in a similar context in \cite{smoothdensektheta} and also in work of Roberts in \cite{roberts1985cyclic}. Namely, the measure $\nu$ can be expressed as a sum \begin{equation} \label{nu-decomp} \nu= \nu_{\Ca} + \nu_{\K}
\end{equation}
where the two measures are mutually singular, there exists an increasing sequence of Beurling-Carleson sets of Lebesgue measure zero $\{F_n\}_{n \geq 1}$ such that $$\lim_{n \to \infty} \nu_\Ca(F_n) = \nu_\Ca(\T),$$ and $$\nu_\K(F) = 0$$ for any Beurling-Carleson set $F$ of Lebesgue measure zero. The part $\nu_\Ca$ plays no role in our approximation problem. However, the support of $\nu_\K$ must necessarily be located on the support of $\Delta$ for approximations by smooth functions to be possible. Moreover, if the conditions \eqref{deltalogint} and \eqref{deltastruct} are satisfied and the mass of $\nu_\K$ is located appropriately in the sense that $\nu_\K(\cup_n E_n) = \nu_\K(\T)$, then indeed functions in the class $\A^n$ are dense in $\hb$. In the final result of this article, we extend this density claim to functions in $\A^\infty$, which sharpens the result found in $\cite{DBRpapperAdem}$. A a more detailed exposition of why this approximation result is close to the best possible also appears in \cite{DBRpapperAdem}.

We remark that the very interesting problem of giving an explicit formula for $h$ supported on any given closed set $E$ such that $C_{h1_E} \in \A$ remains open, and the approach presented here is not applicable.

\section{Construction of an analytic "cut-off" function}

\label{cutoffsec}

We start off by presenting the constructing of a certain analytic function with strong decay properties near a given Beurling-Carleson set. The reason for calling it a \textit{cut-off function}, as in the name of the section, will become clear from a proof of the coming application in \thref{constrKhru}. Our construction is a straightforward adaptation of a technique from \cite{hedenmalmbergmanspaces}, more precisely from Lemma 7.11 of that work. We could have also followed the ideas of \cite{novinger1971holomorphic} or \cite{taylor1970ideals}. The proof is included for the reader's convenience and because the construction is crucial for our development.

\begin{lemma}\thlabel{mainlemma}
Let $E$ be a Beurling-Carleson set, of either zero or positive Lebesgue measure. There exists an analytic function $g: \D \to \D$ such that the function $G(t) := g(e^{it})$ is smooth on $\T \setminus E$, and we have the estimate \begin{equation}
\label{Gest} |G^{(m)}(t)| = o(\dist{e^{it}}{E}^N), \quad e^{it} \to E
\end{equation} for each pair of non-negative integer $N$ and $m$. Here $G^{(m)}$ denotes the $m$:th derivative of $G$ with respect to the variable $t$, and $\dist{\cdot}{\cdot}$ denotes the distance between two closed sets. In particular, if $E$ has Lebesgue measure zero, then $G(t)$ is smooth on $\T$.
\end{lemma}

\begin{proof}
Let $\cup_{n \in \mathbb{N}} A_n = \T \setminus E$ be the complement of $E$ with respect to $\T$. For each subarc $A_n$, we perform the classical \textit{Whitney decomposition} $A_n = \cup_{k \in \mathbb{Z}} A_{n,k}$. More precisely, let $A_{n,0}$ be the arc with the same midpoint as $A_n$ but having one third of the length of $A_n$. For this choice of the length we have $|A_{n,0}| = \dist{A_{n,0}}{E}$. The arcs $A_{n,-1}$ and $A_{n,1}$ should be chosen adjacent to $A_{n,0}$ from the left and right respectively, and their lengths should be chosen, again, such that $|A_{n,-1}| = \dist{A_{n,-1}}{E}$ and $|A_{n,1}| = \dist{A_{n,1}}{E}$. It is easy to see that the correct choice is $|A_{n,1}| = |A_{n,-1}| = \frac{|A_n|}{6}.$  Proceeding in this manner, we will obtain a decomposition \[ \T \setminus E = \cup_{n} A_n = \cup_{n,k} A_{n,k}\] where for each arc $A_{n,k}$ we have \begin{equation} \label{A_nkdist}
|A_{n,k}| = \frac{|A_n|}{3\cdot 2^{|k|}} = \dist{A_{n,k}}{E}
\end{equation} A straight-forward computation based on \eqref{A_nkdist} will show that \begin{gather*} \sum_{n,k} |A_{n,k}| \log(1/|A_{n,k}|) < \infty. \end{gather*} Let $\{B_j\}_j$ be a re-labelling of the arcs $\{A_{n,k}\}_{n,k}$ and \{$\lambda_j\}_j$ a positive sequence tending to infinity such that \[ \sum_j \lambda_j |B_j| \log(1/|B_j|) < \infty. \] Now let $r_j = 1 + |B_j|$, $b_j \in \T$ be the midpoint of the arc $B_j$, and consider the function \begin{equation} \label{hfunc} h(z) = - \sum_{j} h_j(z) =  - \sum_j \frac{\lambda_j b_j|B_j| \log(1/|B_j|)}{r_jb_j - z}, \quad z \in \D. \end{equation} It is not hard to see that the real part of $h(z)$ is negative in $\D$. In fact, the real part of the $j$:th term in the sum is \begin{gather*} -\Re h_j(z) = -\lambda_j |B_j| \log(1/|B_j|)\frac{\Re (r_j - \conj{z}b_j)}{|r_jb_j - z|^2} < 0, \end{gather*} where the last inequality follows from $\Re (r_j - \conj{z}b_j) < 0$, which is a consequence of the inequalities $r_j > 1$ and $|\conj{z}b_j| < 1$. It follows that \begin{equation}\label{geq} g(z) := \exp(h(z))\end{equation} is bounded by 1 in modulus for $z \in \D$. Moreover, the series defining $h(z)$ converges also for $z \in B_j$, and $h_j$ extends analytically across each $B_j$, because the poles $\{r_jb_j\}_j$ of $h$ cluster only at the set $E$. For $z \in B_j$, we have that the quantities $|r_jb_j - z|$ and $\Re(r_j - \conj{z}b_j)$ are both approximately equal to $|B_j|$, and so \[ |g(z)| \leq \exp( -\Re h_j(z)) \leq \exp( - c \lambda_j \log(1/|B_j|))  = |B_j|^{c\lambda_j}\] for some positive constant $c$. Since $|B_j|$ equals the distance from $B_j$ to $E$, for $z \in B_j$ we obtain \[ |g(z)| \leq C \dist{z}{E}^{c\lambda_j} \] for some positive constant $C > 0$ independent of $j$. Note that as $z$ tends to $E$ along the complement $\T \setminus E$, it needs to pass through infinitely many intervals $B_j$. Since $\lambda_j$ tends to infinity, we obtain that \begin{equation} \label{gest}
|g(z)| = o(\dist{z}{E}^N)
\end{equation} as $z \to E$ along the complement of $E$ on $\T$, for any choice of positive integer $N$. 

Clearly $G(t) := g(e^{it})$ is smooth on $\T \setminus E$. On this set, the derivatives $G^{(m)}(t)$ have the form $H(e^{it})G(t)$, where $H$ is a linear combination of products of derivatives of $h(e^{it})$ with respect to $t$. But a glance at \eqref{hfunc} shows that such a product cannot grow faster than a constant multiple of $\dist{e^{it}}{E}^{-n}$ for $e^{it} \in \T \setminus E$, for some integer $n$ depending only on the number of derivatives taken. Together with \eqref{gest}, we see that the claim in the lemma follows. 
\end{proof}

Note the fact that the proof above gives an explicit computable formula for the cut-off function $g$. It is given in terms of the Beurling-Carleson set $E$ and is presented in equations \eqref{hfunc} and \eqref{geq}.

\section{A constructive proof of Khrushchev's theorem}
\label{constrproofsec}

\subsection{Smooth Cauchy transforms} As before, let $E$ be a Beurling-Carleson set of positive measure. \thref{mainlemma} will allow us to construct, and give explicit formulas for, measurable functions supported on $E$ which have a smooth Cauchy transform. Thus we will now give the promised constructive proof of the theorem of Khrushchev from his seminal work \cite{khrushchev1978problem}. We state the theorem in the following somewhat more general form than it is stated in the mentioned work. 

\begin{prop}{\textbf{(Construction of smooth Cauchy transforms)}} \thlabel{constrKhru} Let $E$ be a Beurling-Carleson set of positive measure such that $E \neq \T$, and $w$ be a bounded positive measurable function with support on $E$ which satisfies $\int_E \log(w) d\m > -\infty$. Let $W$ be the outer function \begin{equation} \label{Wformula}
W(z) = \exp\Big( \int_E \frac{z + \conj{\zeta}}{z-\conj{\zeta}}\log(w(\zeta))d\m(\zeta) \Big)
\end{equation} and $g$ be the function associated to $E$ which is given by \thref{mainlemma}. Consider the set \begin{equation}
\label{setK}
K = \Big\{ s = \conj{\zeta p g W} : p \text{ analytic polynomial } \Big\} 
\end{equation} consisting of functions on $\T$, where $\zeta$ is the coordinate function on $\T$. Then the Cauchy transform \begin{equation}
    \label{Cseq} C_{s1_E}(z) := \int_E \frac{s(\zeta)}{1-z\conj{\zeta}}d\m(\zeta)
\end{equation} is a non-zero function in $\A^\infty$ for each non-zero $s \in K$, the restrictions to $E$ of elements of the set $K$ form a dense subset of $L^2(1_E d\m)$, and the set \begin{equation} \label{setCK} C_EK := \{ C_{s1_E} : s \in K \} \end{equation} is dense in $H^2$. 
\end{prop}

Certainly our more general form of the theorem, together with the density statements, is obtainable by Khrushchev's methods in \cite{khrushchev1978problem}. We therefore emphasize that our main contribution in this context are the explicit formulas for the measurable functions supported on $E$ for which the Cauchy transform is an analytic function in $\A^\infty$. More precisely, the formulas for the functions in $K$ are given by the equations \eqref{hfunc}, \eqref{geq} and \eqref{Wformula}. 

The density statements in \thref{constrKhru} will be useful for our further applications. It is not our point to prove these density statements constructively. In this part of the proof, we will use the following well-known theorem.

\begin{lemma}{(\textbf{Beurling-Wiener theorem})} \thlabel{beurling-wiener}
Let $M_{\conj{\zeta}}:L^2(\T) \to L^2(\T)$ be the operator of multiplication by $\conj{\zeta}$. The closed $M_{\conj{\zeta}}$-invariant subspaces of $L^2(\T)$ are of the form \[L^2(1_F d\m) = \{ f \in L^2(\T) : f = 0 \text{ almost everywhere on } \T \setminus F\}\] where $F$ is a measurable subset of $\T$, or of the form \[U \conj{H^2} = \{ U\conj{f} : f \in H^2\}\] where $U$ is a unimodular function.
\end{lemma}

For a proof of the Beurling-Wiener theorem, see \cite{helsonbook}, for instance.

\begin{proof}[Proof of \thref{constrKhru}]
Since $s$ is a conjugate analytic and satisfies $\int_\T s d\m = 0$ we have \[\int_\T \frac{s(\zeta)}{1-z \conj{\zeta}}dm(\zeta) = 0\] for each $z \in \D$. This implies that \begin{equation} \label{flip}
C_{s1_E}(z) = \int_E \frac{s(\zeta)}{1-z\conj{\zeta}}d\m(\zeta) = -\int_{\T \setminus E} \frac{s(\zeta)}{1-z\conj{\zeta}}d\m(\zeta).
\end{equation} Consider now the function $S(t) : = s(e^{it})1_{\T \setminus E}(e^{it})$, where $1_{\T\setminus E}$ is the indicator function of the set $\T \setminus E$. From the formula \eqref{Wformula} for $W$ it is clear that this function extends analytically across $\T \setminus E$, and a simple differentiation argument shows that the derivatives in the variable $t$ of the function $W(e^{it})$ admit a bound \begin{equation}
    \label{Westimate} \Big|\frac{\partial^m }{\partial t^m}W(e^{it}) \Big| \leq C_m \cdot \dist{e^{it}}{E}^{-2m} 
\end{equation} for $e^{it} \in \T \setminus E$. Thus by \eqref{Gest} and \eqref{setK} the derivative in $t$ of any order of $S$ tends to zero as $e^{it}$ tends to $E$ along $\T \setminus E$, and it is not hard to see that the derivatives of $S$ vanish on $E$. Thus $S \in C^\infty(\T)$. It follows that the Fourier coefficients $S_n$ of $S$ satisfy $|S_n| \leq C|n|^{-M}$ for each positive integer $M$ and some constant $C = C(M) > 0$. Obviously then the function $C_{s1_E}(z) = \sum_{n = 0}^\infty S_nz^n$ is in $\A^\infty$. It is non-zero if $s$ is non-zero, because the positive Fourier coefficients cannot vanish for the function $s1_E$ which is identically zero on the set $\T \setminus E$ of positive Lebesgue measure. 

The density in $L^2(1_E d\m)$ of the restrictions to $E$ of elements of the set $K$ is an easy consequence of the invariance of $K$ under multiplication by $\conj{\z}$ and the Beurling-Wiener theorem, \thref{beurling-wiener} above. Indeed, the restriction to $E$ of an element of $K$ is non-zero almost everywhere on $E$, but obviously zero on $\T \setminus E$. It follows from \thref{beurling-wiener} that the closure of $K$ could not be anything else than $L^2(1_E d\m)$.

The set $C_EK$ is certainly contained in $H^2$, and the density in $H^2$ follows from the classical Beurling theorem for the Hardy spaces. More precisely, the set $C_EK$ is invariant under the backward shift operator \begin{equation}
    \label{backshift} f(z) \mapsto \frac{f(z)-f(0)}{z}.
\end{equation} Indeed, we have that \begin{equation} \label{Csbackshift}
\frac{C_{s1_E}(z) - C_{s1_E}(0)}{z} =  \int_E \frac{\conj{\zeta}s(\zeta)}{1-z\conj{\zeta}}d\m(\zeta) = C_{\conj{\zeta} {s1_E}}(z).    
\end{equation} By Beurling's theorem the closure of $C_EK$ is either all of $H^2$, or it coincides with a model space $K_\theta$ of functions which have boundary values on $\T$ of the form $\theta \conj{h}, h \in zH^2$, for some non-zero inner function $\theta$. If we would be in the second case, then there would exist a function $k \in zH^2$ such that on the circle $\T$ we would have the equality $s1_E = C_{s1_E} + \conj{k} = \theta \conj{h} + \conj{k}$, and consequently $\conj{\theta}s1_E \in \conj{H^2}$. This is a contradiction, since $\conj{\theta}s1_E$ vanishes on a set of positive measure. 
\end{proof}

\subsection{A technical improvement}

Sets of the form $K$ as in \eqref{setK} have another useful property, one which will be employed in the coming applications. The property is that the set $C_E K$ defined in \eqref{setCK} is contained in a single Hilbert space consisting purely of functions which are in $\A^\infty$. This applies to many sets similar to $K$, as we shall see next.

More precisely, take the function $s_0 := \conj{\zeta g W} \in K$, i.e, the one where $p = 1$ in \eqref{setK}. The only property of $K$ that we will use in the proof is that it is of the form \[\{ \conj{p}s_0 : p \text{ analytic polynomial} \}.\] For an analytic polynomial $p(z) = \sum_{n=0}^{d} p_nz^n$, define the operator \begin{equation}
    \label{PLeq} p(L) := \sum_{n=0}^{d} p_nL^n,
\end{equation} where $L$ is the backward shift operator $L$ defined in \eqref{backshift}. Every other element of $C_EK$ can be expressed as $p(L)C_{s_01_E}$ for some analytic polynomial. This claim is a consequence of the formula \begin{equation} \label{pLCs} p(L)C_{s_01_E}(z) = \int_{E} \frac{\conj{\zeta p g W}}{1 - z\conj{\zeta}}  d\m(\zeta) \end{equation} which, in turn, is a consequence of \eqref{Csbackshift}. Thus the Taylor coefficients in the family $C_EK$ have similar asymptotic behaviour, and we exploit this fact in the following way. Being a function in $\A^\infty$, the Taylor coefficients $\{ S_k \}_{k=0}^\infty$ of $C_{s_01_E}$ satisfy \begin{equation}
    \sum_{k=0}^\infty k^N |S_k|^2 < \infty
\end{equation} for all positive integers $N$. It follows that for each $N \geq 1$, there exists a positive integer $K(N)$ such that \begin{equation}
    \sum_{k=K(N)}^\infty k^N |S_k|^2 < \frac{1}{2^N}.
\end{equation} Let $K(0) = 0$ and define a sequence $\{\alpha_k\}_{k=0}^\infty$ by \[ \alpha_k = k^N, \quad K(N) \leq k < K(N+1). \] This sequence is increasing, and satisfies \begin{gather} \label{s0wk}
\sum_{k=0}^\infty \alpha_k|S_k|^2 = \sum_{N=0}^\infty \sum_{k=K(N)}^{K(N+1)-1} k^N |S_k|^2 \leq \sum_{N=0}^\infty \frac{1}{2^N} < \infty.
\end{gather} Moreover, since $\alpha_k \geq k^{N+1}$ if $k \geq K(N+1)$, we have that \begin{equation} \label{rapidgrowthwk}
\lim_{k \to \infty} \frac{\alpha_k}{k^N} \geq \lim_{k \to \infty} \frac{k^{N+1}}{k^N} = \infty
\end{equation} for any positive integer $N$.

\begin{definition} A sequence of positive numbers $\bm{\alpha} = \{ \alpha_k\}_{k=0}^\infty$ is \textit{rapidly increasing} if \begin{equation} 
\lim_{k \to \infty} \frac{\alpha_k}{k^N} = \infty
\end{equation} holds for each positive integer $N$.\end{definition}

Thus we have constructed above a rapidly increasing sequence. In the coming application, we will also need the very mild condition \[\lim_{k \to \infty} \alpha_k^{1/k} = 1 \] which we can safely assume. Indeed, by replacing $\alpha_k$ by $\min (\alpha_k, k^{\sqrt{k}})$, we still have a sequence which is rapidly increasing, and moreover \[ 1 \leq \lim_{k\to \infty} \alpha_k^{1/k} \leq  \lim_{k \to \infty} \exp(\log(k)/\sqrt{k}) = 1. \]

We now make a somewhat trivial observation, which will however be important in the sequel. Because the sequence $\bm{\alpha} = \{\alpha_k\}_{k=0}^\infty$ is increasing, it also follows that whenever an analytic function $f$ has a Taylor series which satisfies \eqref{s0wk}, then so does the backward shift $Lf$ of this function. Thus also $p(L)f$ satisfies this property, for all analytic polynomials $p$ (and in fact, so does appropriately defined $h(L)f$ for any bounded analytic function $h$, see \thref{ToeplitzXalpha} below). Using also the formula \eqref{pLCs}, we have proved the following technical result.

\begin{prop}  \thlabel{rapidweight}
Let $s_0$ be a measurable function on $\T$ for which the Cauchy transform $C_{s_0}$ is a function in $\A^\infty$. 
Then there exists a rapidly increasing sequence $\bm{\alpha} = \{\alpha_k\}_{k=1}^\infty$ satisfying \[\lim_{k \to \infty} a_k^{1/k} = 1 \] and such that \[ \sum_{k=0}^\infty \alpha_k|f_k|^2 < \infty \] for all functions $f$ which are Cauchy transforms $f = C_s$ of a function $s$ from the set \[ \{ s = \conj{p} s_0 : p \text{ analytic polynomial } \}. \]
\end{prop}

\section{Weighted sequence spaces}

We will explore the Hilbert spaces implicitly appearing in \thref{rapidweight} a little more, and prove a few basic facts about their duality and operators acting on them. All results in this section are certainly well-known. The main results of the following Sections \ref{uncertaintysec} and \ref{permanencesec} will be stated in the context of these Hilbert spaces.

\subsection{Definition and duality} \label{Xalphaduality} For a sequence of positive numbers $\bm{\alpha} = \{ \alpha_k\}_{k=0}^\infty$, we define the Hilbert space $X(\bm{\alpha})$ to consist of formal power series $f(z) = \sum_{k=0}^\infty f_nz^n$ which satisfy \begin{equation}
    \label{halphanorm}
    \|f\|^2_{X(\bm{\alpha})} := \sum_{k=0}^\infty \alpha_k|f_k|^2 < \infty.
\end{equation} It is obvious that if $\bm{\alpha}$ is rapidly increasing, then $X(\bm{\alpha}) \subset \A^\infty$. We define the \textit{dual sequence} $\bm{\alpha^{-1}}$ by the equation \[ \bm{\alpha^{-1}} := \{ \alpha_k\}_{k=0}^\infty. \] The space $X(\bm{\alpha^{-1}})$ is isometrically isomorphic to the dual space of $X(\bm{\alpha})$ under the pairing which maps $f \in X(\bm{\alpha})$, $g \in X(\bm{\alpha^{-1}})$ to the complex number
\begin{equation} \label{Xalphapairing}
   \ip{f}{g} := \sum_{k=0} f_k\conj{g_k},
\end{equation} where the sequences $\{f_k\}_{k=0}^\infty$ and $\{g_k\}_{k=0}^\infty$ are the coefficients in the formal power series expansions of $f$ and $g$ respectively. 

In fact, for us the spaces $X(\bm{\alpha})$ and $X(\bm{\alpha}^{-1})$ will always be genuine spaces of analytic functions on $\D$. Indeed, a property which ensures this is $\lim_{k \to \infty} \alpha_k^{1/k} = 1$. The sequences appearing in our context will be the ones constructed in \thref{rapidweight} and their dual sequences, so we can safely assume below that this assumption is always satisfied. To see indeed that the assumption $\lim_{k \to \infty} \beta_k^{1/k} = 1$ implies that the radius of convergence of a formal power series $f \in X(\bm{\beta})$ is equal to at least 1, we compute \[ \limsup_{k\to \infty} |f_k|^{1/k} = \limsup_{k \to \infty} \frac{ (\beta_k |f_k|^2)^{1/2k}}{\beta_k^{1/2k}} \leq \limsup_{k \to \infty} \frac{1}{\beta_k^{1/2k}}  = 1,\] where in the next-to-last step we used that \[\lim_{k \to \infty} \beta_k|f_k|^2 = 0,\] so that \[\limsup_{k \to \infty} \, (\beta_k|f_k|^2)^{1/2k} \leq 1. \] 

Finally, an obvious but important property of the presented duality pairing is that if $f$ and $g$ happen to be functions in $H^2$, then we have that \eqref{Xalphapairing} equals \[ \ip{f}{g} = \int_\T f\conj{g} \, d\m. \] In other words, the duality pairing coincides with the usual $L^2(\T)$-duality pairing in the case $f$ and $g$ are functions in $H^2$. We shall often implicitly use this property.

\subsection{Toeplitz operators}

The usual Toeplitz operator with symbol $h \in L^\infty(\T)$ acts on an $H^2$ function $f$ by the formula \[ T_{h} f(z) := \int_{\T} \frac{f(\zeta)h(\zeta)}{1-\conj{\zeta}z}\,d\m(\zeta). \] If $p$ is a polynomial, then \begin{equation}
    \label{toeplitzfoias}
    T_{\conj{p}} f = \tilde{p}(L) f,
\end{equation} where $\tilde{p}(L)$ is defined according to \eqref{PLeq} and where $\tilde{p}(z) = \conj{p(\conj{z})}$. If $h$ is in $H^\infty$, then we can equivalently define the operator $T_h$ as the mapping taking the function $f(z), z \in \D$, to the function $h(z)f(z), z \in \D$. We denote by $M_h$ the multiplication operator \[ M_h f(z) = h(z)f(z), \quad z \in \D\] which acts on the space of all holomorphic functions on $\D$. If $h$ is analytic and $f \in H^2$, then it is well-known that $T_hf = M_hf$. We say that the Toeplitz operator is \textit{co-analytic}, or \textit{has a co-analytic symbol}, if it is of the form $T_{\conj{h}}$ for $h \in H^\infty$.

\begin{prop} \thlabel{ToeplitzXalpha}
Let $\bm{\alpha}$ be a sequence of positive numbers such that the corresponding $X(\bm{\alpha})$ space consistis of analytic functions in $\D$. 
\begin{enumerate}[(i)]
    \item If $\bm{\alpha}$ is increasing, then $X(\bm{\alpha})$ is continuously contained in $H^2$, and the Toeplitz operators with bounded co-analytic symbols are bounded on $X(\bm{\alpha})$.
    \item If $\bm{\alpha}$ is decreasing, then the operators $M_h$ with bounded analytic symbols are bounded on $X(\bm{\alpha})$.
\end{enumerate}

In both cases, the corresponding operators have a norm which is less than or equal to the supremum norm of the corresponding symbol. 
\end{prop}

\begin{proof}
We prove $(i)$. It is clear that $X(\bm{\alpha})$ is continuously contained in $H^2$. A direct computation shows that $\bm{\alpha}$ being increasing implies that the backward shift operator $L$ in \eqref{backshift} is a contraction on $X(\bm{\alpha})$. There certainly exists no subspace of $X(\bm{\alpha})$ on which $L$ acts as a unitary (or even an isometry), so the Nagy-Foias functional calculus (see \cite{nagyfoiasharmop} for details) allows us to define the operator \[h(L): X(\bm{\alpha}) \to X(\bm{\alpha})\] for any bounded analytic function $h$, in such a way that the definition is consistent with \eqref{PLeq} for polynomials $h$, the operator norm of $h(L)$ is at most $\|h\|_{\infty}$, and if \[\lim_{n \to \infty} h_n(\zeta) \to h(\zeta)\] almost everywhere on $\T$ and \[\sup_n \|h_n\|_\infty < \infty,\] then $h_n(L)$ converges in the strong operator topology to $h(L)$. The operators $h(L)$ are co-analytic Toeplitz operators with symbol $\conj{\tilde{h}}$. To see this, fix $h \in H^\infty$ and let $\{h_n\}_n$ be the Fej\'er polynomials for $h$, so that the above properties of the Nagy-Foias functional calculus imply that $h_n(L)f \to h(L)f$ in the norm of $X(\bm{\alpha})$, for any $f \in X(\bm{\alpha})$. The same is true in the norm of $H^2$. Let $\tilde{h}(z) = \conj{h(\conj{z})}$, and recall \eqref{toeplitzfoias}. For $z \in \D$, we get  

\[ h(L)f(z) = \lim_{n \to \infty} h_n(L)f(z) = \lim_{n \to \infty} T_{\conj{\tilde{h_n}}}f(z) = T_{\conj{\tilde{h}}} f(z).
\] Thus $h(L) = T_{\conj{\tilde{h}}}$, and by reversing roles of $h$ and $\tilde{h}$, we see that the Nagy-Foias functional calculus for $L$ on $X(\bm{\alpha})$ is a bijection onto the co-analytic Toeplitz operators.

Next, we prove $(ii)$. If $\bm{\alpha}$ is decreasing, then the dual sequence $\bm{\alpha^{-1}}$ is increasing, so by $(i)$ we can define $T_{\conj{h}}^*: X(\bm{\alpha}) \to X(\bm{\alpha})$ as the adjoint of $T_{\conj{h}}: X(\bm{\alpha^{-1}}) \to X(\bm{\alpha^{-1}})$ with respect to our duality pairing \eqref{Xalphapairing} between the spaces. Let $f \in X(\bm{\alpha})$, $\lambda \in \D$, and $s_\lambda(z) = \frac{1}{1-\conj{\lambda}z}$. Recall that $s_\lambda$ is an eigenvector of $T_{\conj{h}}$, with eigenvalue $\conj{h(\lambda)}$. We compute

\begin{gather*}
    T_{\conj{h}}^*f (\lambda) = \ip{T_{\conj{h}}^*f}{s_\lambda} = \ip{f}{T_{\conj{h}}s_\lambda } = h(\lambda)\ip{f}{s_\lambda} = h(\lambda)f(\lambda).
\end{gather*}
Thus the adjoints of the co-analytic Toeplitz operators on $X(\bm{\alpha^{-1})}$ are multiplication operators on $X(\bm{\alpha)}$. The operator norm of $T_{\conj{h}}^*$ equals to operator norm of $T_{\conj{h}}$, which is at most $\|h\|_\infty$, as was noted in the proof of part $(i)$. The proof is complete.
\end{proof}

\section{Completely non-isometric shifts} 

\label{uncertaintysec}

\subsection{A big Hilbert space with a completely non-isometric shift}

In the next proposition we construct a Hilbert space of analytic functions on $\D$ which has desirable properties and which is strictly larger than any space $\Po^2(\mu)$ with measure $\mu$ being of the form $$d \mu = d\mu_C := (1-|z|^2)^C dA + w d\m$$ and $C$ being any positive number. This Hilbert space will play an important role in the proof of the main result of Section \ref{dbrsec}. Another application is presented in \thref{uncertstrong} below. We note that a similar result certainly can be reached by methods of Khrushchev developed in \cite{khrushchev1978problem}, but our proof below is different, and relies fully on construction of smooth Cauchy transforms.

\begin{prop} \thlabel{XdiagProp}
Let $E$ be a Beurling-Carleson set of positive Lebesgue measure, and $w$ be a bounded positive measurable function which is supported on $E$ and satisfies $\int_E \log(w) d\m > -\infty$. For a sequence $\bm{\alpha}$, consider the product space \[X(\bm{\alpha}^{-1}) \oplus L^2(w \, d\m)\] and the norm closure \[\mathcal{D}(\bm{\alpha}^{-1}, w)\] of the linear manifold \[ \{ (p,p) \in X(\bm{\alpha}^{-1}) \oplus L^2(w \, d\m) : p \text{ analytic polynomial }\}. \] There exists a rapidly increasing sequence $\bm{\alpha} = \{\alpha_k\}_{k=0}^\infty$ such that the space $\mathcal{D}(\bm{\alpha}^{-1}, w)$ has the following property: $f_1 \equiv 0$ implies that $f_2 \equiv 0$, for any tuple $(f_1, f_2) \in \mathcal{D}(\bm{\alpha}^{-1}, w)$.
\end{prop}

\begin{proof}
The proof is very simple in principle. We will use the set $K$ in \eqref{setK} in combination with the sequence constructed in \thref{rapidweight}, and this will provide us with enough functionals on $X(\bm{\alpha}^{-1}) \oplus L^2(w d\m)$ to conclude that $f_1 \equiv 0$ implies $f_2 \equiv 0$, by a straight-forward duality argument involving the Beurling-Wiener theorem. 

For each $s \in K$, consider the functional \begin{equation}  \label{annihilatorFunctionals}p \mapsto -\int_\T p \conj{C_{s1_E}} \,d\m + \int_{E} p \conj{s} \, d\m \end{equation} which we define on the set $\{ (p,p) \in X(\bm{\alpha}^{-1}) \oplus L^2(w d\m) : p \text{ analytic polynomial }\}$. By construction, these functionals are the zero functionals. Apply \thref{rapidweight} to produce a rapidly increasing sequence $\bm{\alpha}$ such that $C_{s1_E} \in X(\bm{\alpha})$ for all $s \in K$. The constructed functionals are then continuous with respect to the metric $X(\bm{\alpha}^{-1}) \oplus L^2(w d\m)$. Indeed, we see from \eqref{setK} that $s = wq$ on $E$, where $q$ is a bounded function, and so \[ \int_E p\conj{s} \, d\m \leq C \|p \sqrt{w}\|_{L^2}, \] by Cauchy-Schwarz inequality. 

Now let $(f_1, f_2)$ lie $\mathcal{D}(\bm{\alpha}^{-1}, w)$ and assume that $f_1 \equiv 0$. Fix a sequence of polynomials $\{p_n\}_{n=1}^\infty$ such that $(p_n, p_n) \to (f_1, f_2) = (0, f_2)$ in the norm of $X(\bm{\alpha}^{-1}) \oplus L^2(w d\m)$. Then $(0, f_2)$ is annihilated by any functional in \eqref{annihilatorFunctionals}, and so \[ \int_E f_2\conj{s} d\m = 0, \quad s \in K. \] By the density statements in \thref{constrKhru}, we conclude that $f_2 \equiv 0$. 
\end{proof}

Let us take another look at the space $\mathcal{D}(\bm{\alpha}^{-1}, w)$ appearing above, assuming that it is satisfying the conclusion of \thref{XdiagProp}. If $(f, f_1)$ and $(f, f_2)$ are two tuples in $\mathcal{D}(\bm{\alpha}^{-1}, w)$ with coinciding first coordinate, then $(0, f_1 - f_2) \in \mathcal{D}(\bm{\alpha}^{-1}, w)$, and the above result implies that $f_1 \equiv f_2$. In particular, the projection $(f,f_1) \mapsto f$ onto the first coordinate is an injective mapping from such tuples to analytic functions on $\D$. But this means that $\mathcal{D}(\bm{\alpha}^{-1}, w)$ is in essence a space of analytic functions in which the analytic polynomials are dense. 

We make three more very simple but important observations.

\begin{prop} \thlabel{Dmult}
Let $\mathcal{D}(\bm{\alpha}^{-1}, w)$ be as in \thref{XdiagProp}, and identify it with a space of analytic functions on $\D$ as described above. 

\begin{enumerate}[(i)]
    \item $M_z: \mathcal{D}(\bm{\alpha}^{-1},w) \to \mathcal{D}(\bm{\alpha}^{-1},w)$ is completely non-isometric.
    \item If $f \in H^2$, then $f \in \mathcal{D}(\bm{\alpha}^{-1}, w)$ and the corresponding tuple equals $(f,f)$, where in the second coordinate $f$ is interpreted in the sense of boundary values of $f$ on $\T$.
    \item Every bounded analytic function $h$ defines a multiplication operator $M_h$ on $\mathcal{D}(\bm{\alpha}^{-1}, w)$, with norm at most $\|h\|_\infty$.
    
\end{enumerate}
\end{prop}

\begin{proof}
Part $(i)$ follows from the paragraph above. The only way a function $f \in \mathcal{D}(\bm{\alpha}^{-1},w)$ satisfies $\|M_z f\|_{\mathcal{D}(\bm{\alpha}^{-1},w)} = \|f\|_{\mathcal{D}(\bm{\alpha}^{-1},w)}$ is if $f$ vanishes on $\D$, which does not happen by \thref{XdiagProp}.

Part $(ii)$ follows in a similar way. We need to note only that since $\bm{\alpha^{-1}}$ is decreasing and $w$ is bounded, then for a suitable sequence $\{p_n\}_n$ of Taylor polynomials of $f \in H^2$ the tuples $(p_n, p_n)$ will converge in the norm of $X(\bm{\alpha}^{-1}) \oplus L^2(w \, d\m)$ to $(f,f)$. By part $(i)$, or the discussion in the paragraph above, there is only one tuple in $\mathcal{D}(\bm{\alpha}^{-1},w)$ which has $f$ as the first coordinate. So the tuple representing $f \in H^2 \cap \mathcal{D}(\bm{\alpha}^{-1},w)$ is precisely $(f,f)$.

To prove part $(iii)$, let $h$ be a bounded analytic function and $\{h_n\}_n$ its Fej\'er means. Let $\{p_n\}_n$ be a sequence of polynomials converging to $f \in \mathcal{D}(\bm{\alpha}^{-1}, w)$. Then \thref{ToeplitzXalpha} implies that $\{h_np_n\}_n$ is a norm-bounded sequence in $\mathcal{D}(\bm{\alpha}^{-1}, w)$. The weak limit of this sequence equals $hf \in \mathcal{D}(\bm{\alpha}^{-1}, w)$, and \[ \|hf\|_{\mathcal{D}(\bm{\alpha}^{-1}, w)} \leq \liminf_{n \to \infty} \|h_np_n\|_{\mathcal{D}(\bm{\alpha}^{-1}, w)} \leq \|h\|_\infty \|f\|_{\mathcal{D}(\bm{\alpha}^{-1}, w)}.\] The last inequality is again a consequence of \thref{ToeplitzXalpha}.

\end{proof}

\subsection{An uncertainty-type principle}

The singificance of \thref{XdiagProp} might not be easy to appreciate. In this section, which is independent of the rest of the article, we want to highlight how such a result can be applied in the theory of $\Po^2(\mu)$-spaces and how it relates to a classical result in the theory of Hardy spaces.

Recall that a square integrable function $f$ on $\T$ which lies in the closure of the analytic polynomials (that is, in the Hardy space $H^2$) cannot vanish on a set of positive measure. Thus the spectral smallness of $f$ (vanishing of negative Fourier coefficients of $f$) implies that the function cannot be too small. A beautiful exposition of this result, and other manifestations of the uncertainty principle in harmonic analysis, can be found in \cite{havinbook}. 
A combination of the deep work of Aleman, Richter and Sundberg in \cite{aleman2009nontangential} and our \thref{XdiagProp} will establish the following result of similar nature.

\begin{cor} \textbf{(An uncertainty-type principle for a class of $\Po^2(\mu)$-spaces)} \thlabel{uncertstrong} Let $C > -1$ and $E$ be a Beurling-Carleson set of positive measure. Let $w$ be a bounded positive measurable function which is supported on $E$ and satisfies $\int_E \log(w) d\m > -\infty$. Consider the measure $$d\mu = (1-|z|^2)^C dA + w d\m$$ and the classical Lebesgue space $L^2(\mu)$. Let $\Po^2(\mu)$ be the closure of analytic polynomials in $L^2(\mu)$. Then we have that $f \neq 0$ almost everywhere with respect to $\mu$, for any non-zero $f \in \Po^2(\mu)$.
\end{cor}

\begin{proof}
A computation shows that for $f(z) = \sum_{k=0}^\infty f_kz^k$, we have $$\int_{\D} |f(z)|^2  (1-|z|^2)^C dA(z) = \sum_{k=0} \beta_k(C) |f_k|^2$$ where the weights $\beta_k(C)$ satisfy  the asymptotics $$\beta_k(C) \simeq \frac{1}{k^{C+1}}.$$ This means that convergence of polynomials $\{p_n\}_n$ in $\Po^2(\mu)$ implies convergence of $(p_n, p_n)$ in the space $X(\bm{\alpha}^{-1}) \oplus L^2(w d\m)$ appearing in \thref{XdiagProp}, and a direct consequence is that $\Po^2(\mu)$ contains no non-zero function which vanishes on $\D$. In particular, $\Po^2(\mu)$ does not contain the characteristic function of any subset of $\T$ of positive measure, and every element $f \in \Po^2(\mu)$ has a unique restriction $f|\D$ to $\D$, which of course is an analytic function. In particular, the space satisfies the assumptions of \cite[Theorem A]{aleman2009nontangential}, and the conclusion of that theorem is that for any $f \in \Po^2(\mu)$, its restriction $f|\D$ has a non-tangential limit almost everywhere with respect to $\mu|\T$, and this limit agrees almost everywhere with $f|\T$. If $f$ would vanish on a set of positive $\mu|\T$-measure, then a classical theorem of Privalov (see \cite{koosis}, for instance) can be used to deduce that $f \equiv 0$ throughout $\cD$. 
\end{proof}

In the above result we can obviously replace the part $d\mu|\T = w d\m$ with a more general weight $w$ which is carried by a countable union $\{E_n\}_n$ of Beurling-Carleson sets of positive measure, and where the weight $w$ is $\log$-integrable on each set $E_n$ separately. 

We want to remark also that the use of the very deep and general Aleman-Richter-Sundberg theorem from \cite{aleman2009nontangential} in the above proof can likely be avoided, and the existence of non-tangential limits on $E$ for functions $f$ in $\Po^2(\mu)$ of the described form, or even in the space $\mathcal{D}(\bm{\alpha}^{-1},w)$, is likely accessible in a more straightforward way (see the introductory section of the article \cite{aleman2009nontangential} for an exposition of previously attained special cases of the Aleman-Richter-Sundberg theorem). 

\section{A permanence principle for inner factors}

\label{permanencesec}

Let $\mathcal{H}$ be a space of analytic $\D$ which includes at least $H^\infty$, the algebra of bounded analytic functions. A situation which appears in context of the duality approach to certain approximation problems in function theory (see \cite{limani2021abstract} and \cite{DBRpapperAdem}), and which is certainly also of independent interest, is the following. Assume that $\hil$ carries a norm (or at least some other type of topological structure) and we have a convergent sequence of the form $$\lim_{n\to \infty} \|\theta f_n-f\|_{\mathcal{H}} = 0,$$ where $\theta$ is an inner function, and all other appearing functions are bounded and analytic in $\D$. Then, in particular, $f$ admits an inner-outer factorization $f = IU$ into an inner function $I$ and an outer function $U$. We ask: is $I$ divisible by $\theta$? In other words, does the inner factor $\theta$ get passed onto the limit $f \in H^\infty$ in the metric induced by the norm $\| \cdot \|_\hil$? We will call this property \textit{permanence of an inner function $\theta$} in the corresponding metric.

\begin{definition} \thlabel{permprinc} Let $\hil$ be a topological space of analytic functions which contains $H^\infty$, and $\theta$ be a given inner function. We say that the pair $(\mathcal{H}, \theta)$ satisfies the \textit{permanence property} if $$\lim_{n\to \infty} \theta f_n = f,$$ in the sense of the topology of $\hil$, implies that $f/\theta$ is bounded, whenever $f_n,f$ are bounded analytic functions. \end{definition}

Every inner function $I$ is the form $I = BS_\nu$, where $B$ is a Blaschke product and $S_\nu$ is a singular inner function. The above problem is of course most interesting for singular inner functions, since the Blaschke product $B$ will be passed onto the limit under any reasonable norm defined on analytic functions. In the context of the usual $L^2$-norm computed on the circle, it is of course well-known that any inner function $\theta$ satisfies the permanence property, but for many other metrics a more interesting situation occurs. Here a principal set of examples consists of the weighted $L^2$ metrics on the unit disk $\D$. Recall that a singular inner function has the form
\begin{equation} \label{innerformula} S_\nu(z) = \exp\Big(-\int_\T \frac{\zeta + z}{\zeta - z} d\nu(\z)\Big), \quad z \in \D, 
\end{equation} where $\nu$ is a finite positive singular Borel measure on $\T$. In Section \ref{introsec}, a decomposition of the measure $\nu$ was introduced in \eqref{nu-decomp}. The part $S_{\nu_\Ca}$ is passed onto the limit under convergence of bounded functions in the weighted Bergman spaces norms with polynomially decreasing weights. That is, if $\nu = \nu_\Ca$ in \eqref{nu-decomp}, then for $\theta = S_\nu$ we have that \begin{gather*}
    \lim_{n \to \infty} \int_{\D} |\theta f_n - f|^2 (1-|z|^2)^C dA(z) = 0 \\ \Rightarrow f/\theta \in H^\infty
\end{gather*} whenever $f_n, f$ are all bounded analytic functions, and $C$ is any positive number. In contrast, $S_{\nu_{\K}}$ can vanish under the same circumstances. A proof for the first claim appears in \cite{smoothdensektheta}, while the second is a consequence of a deep cyclicity theorem for inner functions which was independently established by Roberts in \cite{roberts1985cyclic} and Korenblum in \cite{korenblum1975extension}.  

Let us go back to the setting of \thref{XdiagProp} and \thref{Dmult} where the Hilbert space $\mathcal{D}(\bm{\alpha}^{-1}, w)$ appears. We noted that if $\bm{\alpha}$ is suitably chosen, then $\mathcal{D}(\bm{\alpha}^{-1}, w)$ is in fact a space of analytic functions, and it contains $H^2$. Thus the above question of inner factor permanence makes sense in the context of the norm on $\mathcal{D}(\bm{\alpha}^{-1}, w)$. 

Weaker versions of the following results appear in \cite{ptmuinnner}, where circumstances allow for statements in much less technical form. This is a consequence of the fact that the sequences $\bm{\alpha}$ which appear in \cite{ptmuinnner} increase only polynomially. Below, we show that by fixing some singular inner function $\theta$ of some particular structure, we can alter the methods in \cite{ptmuinnner} and construct a space $\mathcal{D}(\bm{\alpha}^{-1}, w)$ where the sequence $\bm{\alpha}$ is rapidly increasing weight and in which the permanence principle in \thref{permprinc} holds for that given $\theta$. The corresponding results for $\Po^2(\mu)$-spaces from \cite{ptmuinnner} are corollaries (we state them in \thref{ptmupermprinc} below), but we will need the full strength of the results established below in the principal application to come. 

\begin{lemma} \thlabel{innerpermanenceCarleson} 
Let $\theta = S_\nu$ be a fixed singular inner function for which in the decomposition \eqref{nu-decomp} of $\nu$, the part $\nu_{\Ca}$ is supported on a single Beurling-Carleson set $F$ of Lebesgue measure zero, and $\nu_\K \equiv 0$. Then there exists a rapidly increasing sequence $\bm{\alpha} = \{\alpha_k\}_{k=0}^\infty$ (which depends on $\theta$) such that the pair $(X(\bm{\alpha}^{-1}), \theta)$ satisfies the permanence property in \thref{permprinc}. 
\end{lemma}

\begin{proof}
The idea of the proof is as follows. Let $u$ be a function in $K_\theta$, the orthogonal complement of $\theta H^2$ in $H^2$, and $\Lambda_u$ be the (in general unbounded with respect to the norm on $X(\bm{\alpha^{-1}})$) linear functional \begin{equation}
    \label{lambdafunc} \Lambda_u f : = \int_{\T} f\conj{u} \, d\m
\end{equation} which is defined for $f \in H^2 \subset X(\bm{\alpha^{-1}})$. If $\Lambda_u$ can be extended to a bounded linear functional on $X(\bm{\alpha^{-1}})$ for $u$ in a dense subset of $K_\theta$, then $\|\theta f_n - f\|_{X(\bm{\alpha}^{-1})} \to 0$, with $f_n, f \in H^2$, will imply that $f$ is orthogonal to $K_\theta$ in $H^2$. Indeed, in such a case we will have $$\ip{f}{u} = \lim_{n \to \infty} \ip{\theta f_n}{u} = 0,$$ for all $u$ in a dense subset of $K_\theta$, and so $f \in (K_\theta)^\perp = \theta H^2$. This of course means that $f/\theta \in H^2$. We will show that such a dense set can be constructed under the stated assumption, for some rapidly increasing sequence $\bm{\alpha}$. The proof will involve construction of a new set of smooth Cauchy transforms similar to those in \eqref{setCK}, and an application of \thref{rapidweight}. 

Note that $\theta$ extends analytically across the set $\T \setminus F$, and a simple differentiation argument and the formula \eqref{innerformula} shows that we have the following estimate: \begin{equation}
    \label{thetaestimate} \Big|\frac{\partial^m }{\partial t^m}\theta(e^{it}) \Big| \leq C_m \cdot \dist{e^{it}}{F}^{-2m}, \quad e^{it} \in \T \setminus F.
\end{equation} Let $g = g_F$ be the function decaying rapidly near $F$ which is given by \thref{mainlemma}. We conclude, similarly to as in the proof of \thref{constrKhru}, that the set of functions on $\T$ defined by \[K_1 := \{ \theta \conj{\zeta p g_F} : p \text{ analytic polynomial } \} \] consists of functions in $C^\infty(\T)$, and thus the Cauchy transform of any function in this set is in $\A^\infty$. Let $P_+$ denote the projection operator from $L^2(\T)$ to the Hardy space $H^2$. Then $P_+f = C_f$, interpreted as functions on the circle. We now verify that these Cauchy transforms are  members of $K_{\theta}$. Let $\ip{\cdot}{\cdot}_{L^2}$ denote the usual inner product for $L^2(\T)$. For $s = \theta \conj{\zeta p g_F}$ and any $h \in H^2$, we have \[\ip{\theta h}{C_s}_{L^2} = \ip{\theta h}{P_+ s}_{L^2} = \ip{\theta h}{s}_{L^2} = \int_\T h \zeta p g_F d\m = 0,\] where the last integral vanishes because the integrand represents the boundary function of an analytic function with a zero at the origin. Thus $C_{s} \in K_{\theta}$ for any $s \in K_{\theta}$. We now verify that this set of Cauchy transforms is dense in $K_{\theta}$. If $f \in K_{\theta}$, then $f\conj{\theta} = \conj{\zeta f_0}$ as boundary functions, where $f_0 \in H^2$. Orthogonality of $f \in K_{\theta}$ to all functions $C_s, s\in K_1$, means that \[ \ip{f }{C_s}_{L^2} = \int_\T \conj{f_0} p g_F d\m = 0.\] Since $g_F$ is outer, the set \begin{equation} \label{pgFset}
    \{ pg_F : p \text{ analytic polynomial } \}
\end{equation} is dense in $H^2$, and so the above implies $f_0 \equiv 0$, which means that $f \equiv 0$. We have thus constructed a dense set of functions in $K_{\theta}$ to which \thref{rapidweight} applies, and the conclusion is that the Cauchy transforms we constructed are all contained in some space $X(\bm{\alpha})$ defined by a rapidly increasing sequence $\bm{\alpha}$. Then the space $X(\bm{\alpha^{-1}})$ satisfies the permanence principle for $\theta$, by the observation in the first paragraph of this proof. \end{proof}

\begin{lemma} \thlabel{innerpermanenceKorenblum} 
Let $E$ be a Beurling-Carleson set of positive Lebesgue measure, and let $w$ be a weight supported on $E$ and satisfying $\int_E \log (w) \, d\m > -\infty$. Let $\theta = S_\nu$ be a fixed singular inner function for which $\nu$ is supported on the set $E$. There exists a rapidly increasing sequence $\bm{\alpha} = \{\alpha_k\}_{k=0}^\infty$ (which depends on $E, w$ and $\theta$) for which the conclusion of \thref{XdiagProp} holds, and moreover the pair $(\mathcal{D}(\bm{\alpha}^{-1}, w) , \theta)$ satisfies the permanence property in \thref{permprinc}. 
\end{lemma}

The difference from \thref{innerpermanenceCarleson} is that \thref{innerpermanenceKorenblum} also applies to the case when $\nu_\K$ in \eqref{nu-decomp} is non-zero.

\begin{proof}
We follow the same idea as in the proof of \thref{innerpermanenceCarleson}. From the weight $w$ we construct the outer function $W$ given by the formula \eqref{Wformula}. For the set $E$ we construct the corresponding function $g = g_E$ as in \thref{mainlemma}, and we define 
\[K_2 := \{ \theta \conj{\zeta p g_E W} : p \text{ analytic polynomial } \}. \] This time, the Cauchy transforms of the functions in $K_2$ are not necessarily smooth. However, they are again contained and dense in $K_\theta$, as in \thref{innerpermanenceCarleson}. The only difference in the proof, which we skip, is that the set in \eqref{pgFset} is replaced by \[ \{ pg_E W : p \text{ analytic polynomial } \},\] which is dense in $H^2$ by the fact that $W$ and $g_E$ are outer. 

For $$s_0 := \theta \conj{\zeta g_E W} \in K_2$$ we define the Cauchy transform $u_0 = C_{s_0}$. We can decompose $u_0$ according to  \begin{gather}
    u_0(z) = \int_{\T} \frac{s_0(\zeta)}{1-\conj{\zeta}z} \, dm(\zeta) = \int_{\T \setminus E} \frac{s_0(\zeta)}{1-\conj{\zeta}z} \, dm(\zeta) + \int_{E} \frac{s_0(\zeta)}{1-\conj{\zeta}z} \, dm(\zeta) \nonumber \\ := u_1(z) + u_2(z)  \label{udecomp}
\end{gather}
Estimates of the form \eqref{Westimate} and \eqref{thetaestimate} show that in $s_01_{\T \setminus E}$ is a function in $C^\infty$, and thus $u_1 \in \A^\infty$. Consequently, by \thref{rapidweight}, there exists a rapidly increasing sequence $\bm{\beta}$ such that $C_{s1_{\T \setminus E}}$ is in $X(\bm{\beta^{-1}})$, for all $s \in K_2$. Apply now \thref{rapidweight} to $E$ and $w$ to obtain another rapidly increasing sequence $\bm{\gamma}$ such that the conclusion of that proposition holds, and let $\bm{\alpha}$ be the termwise minimum of $\bm{\beta}$ and $\bm{\gamma}$: \[ \alpha_k = \min \{ \beta_k, \gamma_k \}, \quad k \geq 0. \] Then $\bm{\alpha}$ is again a rapidly increasing sequence, conclusion of \thref{rapidweight} holds, and $C_{s1_{\T \setminus E}}$ is in $X(\bm{\alpha})$, for all $s \in K_2$.

Moreover, the linear functional \[ f \mapsto \int_\T f\conj{u_2} \, d\m,  \] defined on analytic polynomials $f$, is bounded in the metric of $L^2(w \,d\m)$. Indeed, recall that $|W| = w$ on the set $E$, and so we have \begin{gather*}
    \int_\T f\conj{u_2} \, d\m = \ip{f}{P_+ s_01_E}_{L^2} = \int_E f \conj{s_0} \, d\m = \int_E f \conj{\theta}\zeta g_E W \, d\m \leq C \|f \sqrt{w}\|_{L^2}
\end{gather*} since $\conj{\theta} \zeta g_E$ is bounded. The same argument shows also that $C_{s1_E}$ defines a bounded linear functional on the analytic polynomials in the metric of $L^2(w \, d\m)$, for all $s \in K_2$.  

We let $v = C_s$ for $s \in K_2$, $v_1 = C_{s1_{\T \setminus E}}, v_2 = C_{s1_{E}}$, so that $v = v_1 + v_2$, and go back to the definition of the functional $\Lambda_u$ in \eqref{lambdafunc}. We have just verified that we can decompose it according to \eqref{udecomp}
\begin{equation}
   \Lambda_v f : = \int_{\T} f\conj{v} \, d\m = \int_{\T} f \conj{v_1} \, d\m + \int_{\T} f \conj{v_2} \, d\m.
\end{equation} in such a way that the first piece defines a continuous linear functional on the analytic polynomials in the metric of $X(\bm{\alpha^{-1}})$, and the second piece defines a continuous linear functional on the analytic polynomials in the metric of $L^2(w \, d\m)$. But then these functionals extend continuously to $\mathcal{D}(\bm{\alpha^{-1}}, w)$, and by the density of $\{ C_s : s \in K_2 \}$ in $K_\theta$ and the argument in the first paragraph of the proof of \thref{innerpermanenceCarleson}, we conclude that the pair $(\mathcal{D}(\bm{\alpha}^{-1}, w) , \theta)$ satisfies the permanence property.
\end{proof}

\begin{cor} \thlabel{finalcorollary} \textbf{(An inner factor permanence principle)}
Let $E$ be a Beurling-Carleson set of positive Lebesgue measure, and let $w$ be a weight supported on $E$ and satisfying $\int_E \log (w) \, d\m > -\infty$. Let $\theta = S_\nu$ be a fixed singular inner function such in the decomposition \eqref{nu-decomp}, the part $\nu_{\Ca}$ satisfies $\nu_{\Ca}(F) = \nu_{\Ca}(\T)$ for a single fixed Beurling-Carleson set $F$ of Lebesgue measure zero, and $\nu_\K$ is supported on $E$. There exists a rapidly increasing sequence $\bm{\alpha} = \{\alpha_k\}_{k=0}^\infty$ (which depends on $E, w$ and $\theta$) for which the conclusion of \thref{XdiagProp} holds, and moreover the pair $(\mathcal{D}(\bm{\alpha}^{-1}, w) , \theta)$ satisfies the permanence property in \thref{permprinc}. 
\end{cor}

\begin{proof}
The required rapidly increasing sequence $\bm{\alpha}$ is the one obtained by constructing the termwise minimum of the sequences given by \thref{innerpermanenceCarleson} and \thref{innerpermanenceKorenblum}.
\end{proof}

The above result is essentially optimal. Indeed, if $\nu = \nu_\K$ in \eqref{nu-decomp} and $\nu(\T \setminus E) > 0$, then $S_\nu$ will be divisible by an inner function which is cyclic in $\mathcal{D}(\bm{\alpha}^{-1}, w)$ (and so certainly cannot satisfy the permanence property). This can be seen from the corresponding cyclicity result in \cite{ptmuinnner} for the class of $\Po^2(\mu)$-spaces appearing in \thref{uncertstrong}. In the other direction, we note that the main inner factor permanence result in \cite{ptmuinnner} is an immediate consequence of \thref{finalcorollary}. Here is the statement.

\begin{cor} \textbf{(The permanence principle for a class of $\Po^2(\mu)$-spaces)} \thlabel{ptmupermprinc} Let $C > -1$ and $E$ be a Beurling-Carleson set of positive measure. Let $w$ be a bounded positive measurable function which is supported on $E$ and satisfies $\int_E \log(w) d\m > -\infty$. Consider the measure $$d\mu = (1-|z|^2)^C dA + w d\m$$ and the classical Lebesgue space $L^2(\mu)$. Let $\Po^2(\mu)$ be the closure of analytic polynomials in $L^2(\mu)$. If $\theta = S_\nu$ be a singular inner function such that in the decomposition \eqref{nu-decomp} the part $\nu_\K$ is supported on $E$, then the pair $(\Po^2(\mu), \theta)$ satisfies the permanence property in \thref{permprinc}. 
\end{cor}

We skip the proof, which is similar to the proof of \thref{uncertstrong}. 

\section{Density of smooth functions in extreme $\hb$ spaces}

This final section is devoted to the proof of density of smooth functions in the class of de Branges-Rovnyak spaces described in Section \ref{introsec}.

\label{dbrsec}

\subsection{A little background on $\hb$}

The following construction of the $\hb$ space appears in \cite{dbrcont}. 

\begin{prop} \thlabel{normformula}
Let $b$ be an extreme point of the unit ball of $H^\infty$, $$E = \{ \zeta \in \T : |b(\zeta)| < 1 \},$$ and let $\Delta = \sqrt{1-|b|^2}$ be a function on the circle $\T$, defined in terms of boundary values of $b$ on $\T$. For $f\in \Hb$
the equation
\begin{equation} \label{hbconteq} P_+ \conj{b}f = -P_+ \Delta g \end{equation} has a unique solution $g\in L^2(E)$, and the map $J:\Hb\to H^2\oplus L^2(E)$  defined by $$Jf=(f,g),$$ 
is an isometry. Moreover,  \begin{equation} \label{Jort} J(\Hb)^\perp = \Big\{ (bh, \Delta h) : h \in H^2 \Big\}. \end{equation}
\end{prop}

The benefit of the above described way of constructing the $\hb$ space (that is, using the embedding $J$ above, and an orthogonal complement) is that it will be particularly easy to implement our duality argument. 

We need only one more lemma before going into the final proof.

\begin{lemma} \thlabel{densereductionlemma}
Let $b = \theta u$ be an extreme point of the unit ball of $H^\infty$, where $\theta$ and $u$ are the inner and outer factors of $b$, respectively. Further, let $\{\theta_n\}_n$ be a sequence of inner divisors of $\theta$ such that \[ \lim_{n \to \infty} \theta_n(z) = \theta(z)\]  for all $z \in \D$, and let $\{E_n\}_n$ be a sequence of subsets of $\T$ such that \[ E := \{ \zeta \in \T : |b(\zeta)| < 1 \} = \cup_n E_n\]  up to a set of Lebesgue measure zero. For $n \geq 1$, let $u_n$ be the outer function with modulus \[ |u_n| = 1_{E \setminus E_n} + |b|1_{E_n} \] on $\T$. Set $b_n = \theta_n u_n$. Then $\hil(b_n)$ is contractively contained in $\hb$, and $\cup_n \hil(b_n)$ is norm-dense in $\hb$.
\end{lemma}

\begin{proof}
Let $k_b$ and $k_{b_n}$ be the reproducing kernels of $\hb$ and $\hil(b_n)$, respectively. Note that the assumptions imply that $b_n$ divides $b$, in the sense that \begin{equation}
    \label{bbn} |b(z)/b_n(z)| \leq 1, \quad z \in \D.
\end{equation} Then \[ k_b(\lambda,z) - k_{b_n}(\lambda, z) = \frac{\conj{b_n(\lambda)} b_n(z) - \conj{b(\lambda)}b(z)}{1-\conj{\lambda}z} = \conj{b_n(\lambda)}b_n(z)\frac{1 - \conj{b/b_n(\lambda)}b/b_n(z)}{1-\conj{\lambda}z} \] is clearly a positive definite kernel, so by standard theory of reproducing kernel Hilbert spaces (see, for instance, \cite{aronszajn1950theory} or \cite{mccarthypick}) it follows that $\hil(b_n)$ is contractively contained in $\hb$. Moreover, contractivity of the containment means that \[ \|k_{b_n}(\lambda, \cdot)\|^2_{\hb} \leq \|k_{b_n}(\lambda, \cdot)\|^2_{\hil(b_n)} = \frac{1-|b_n(\lambda)|^2}{1-|\lambda|^2} \leq \frac{1-|b(\lambda)|^2}{1-|\lambda|^2}, \] where in the last step we used \eqref{bbn}. So for fixed $\lambda$, the functions $k_{b_n}(\lambda, \cdot )$ are norm-bounded in $\hb$. It is not hard to see from the usual construction of the outer functions that $u_n(z) \to u(z)$ as $n \to \infty$, for every $z \in \D$. Consequently $b_n(z) \to b(z)$ for each $z \in \D$, and even \[ \lim_{n \to \infty} k_{b_n}(\lambda,z) = k_b(\lambda,z), \quad z,\lambda \in \D. \] Together with the norm estimate above, this means that for any fixed $\lambda \in \D$, a suitable subsequence of the kernels $k_{b_n}(\lambda, \cdot)$ will converge weakly in $\hb$ to $k_b(\lambda, \cdot)$. Elementary functional analysis now ensures that $\cup_n \hil(b_n)$ is dense in $\hb$. 
\end{proof}

We remark that a simple consequence of the contractive containment of $\hil(b_n)$ in $\hb$ is the following: density of $\A^\infty \cap \hil(b_n)$ in $\hil(b_n)$ for each $n$ implies density of $\A^\infty \cap \hil(b)$ in $\hb$.

\subsection{The density theorem}

In the below proof, we use the duality pairing $\ip{\cdot}{\cdot}$ appearing in Section \ref{Xalphaduality}. For a set $S$ in either $X(\bm{\alpha})$ or $X(\bm{\alpha^{-1}})$, we denote by $S^\perp$ the linear space of elements in the other space which is annihilated by $S$ under the duality. Basic Hilbert space theory says that $(S^\perp)^\perp$ is the norm-closure of $S$.

\begin{thm} \textbf{($\A^\infty$-density theorem in $\Hb$-spaces)}
Let $b = \theta u$ be an extreme point of the unit ball of $H^\infty$ satisfying the following assumptions. 

\begin{enumerate}[(i)]
    \item There exists an increasing sequence $\{E_n\}_n$ of Beurling-Carleson sets of positive measure such that, up to a set of Lebesgue measure zero, we have the equality \[ E :=  \{ \zeta \in \T : |b(\zeta)| < 1 \} = \cup_n E_n \]
    and \[ \int_{E_n} \log (1-|b|^2) \, d\m > -\infty, \quad \text{ for all } n. \] 
    \item If $\theta = BS_\nu$, where $B$ is a Blaschke product and $\nu$ is the measure defining the singular inner factor as in \eqref{innerformula}, then in the decomposition \eqref{nu-decomp} the part $\nu_\K$ which vanishes on Beurling-Carleson sets of Lebesgue measure zero satisfies $\nu_\K(\T \setminus E) = 0$.
\end{enumerate}

Then $\A^\infty \cap \hb$ is norm-dense in $\hb$.
\end{thm}

Before going into the proof, we remind the reader of what was remarked in Section \ref{introsec}, that the condition $\nu_\K( \T \setminus E) = 0$ appearing in $(iii)$ above is essentially necessary, and that some type of structure condition on the set $E$, and a size condition on the weight $(1-|b|^2)$, also are necessary (examples are mentioned in \cite{DBRpapperAdem}). 

The proof below is essentially the same as the one given in the context of $\Po^2(\mu)$-spaces in \cite{DBRpapperAdem}.

\begin{proof}
By \thref{densereductionlemma} and the remark following it, we can assume that $E$ is a Beurling-Carleson set of positive measure, and the inner factor has a factorization $$\theta = BS_\nu = BS_{\nu_\Ca}S_{\nu_\K},$$ where $\nu_\Ca$ is supported on a single Beurling-Carleson set $F$ of Lebesgue measure zero.

We set $w := \Delta^2 = (1-|b|^2)1_E,$ and apply \thref{finalcorollary} to the data $E, F, S_{\nu_\Ca}, S_{\nu_\K}$ to obtain a rapidly increasing sequence $\bm{\alpha}$ and a space $\mathcal{D}(\bm{\alpha}^{-1}, w)$ which satisfies the conclusion of \thref{finalcorollary}. Since $\bm{\alpha}$ is rapidly increasing, the space $X(\bm{\alpha})$ consists of functions in $\A^\infty$. We will show that $X(\bm{\alpha}) \cap \hb$ is dense in $\hb$. 

Assume that $f \in \hb$ is orthogonal to all functions in $X(\bm{\alpha}) \cap \hb$. In terms of the embedding $J$ appearing in \thref{normformula}, this means that the tuple $Jf := (f,g)$ is orthogonal in $H^2 \oplus L^2(E)$ to all tuples in $J(X(\bm{\alpha}) \cap \hb)$. By the definition of the duality pairing appearing in Section \ref{Xalphaduality}, this means that \[(f,g) \in X(\bm{\alpha^{-1}}) \oplus L^2(E)\] annihilates \[J(X(\bm{\alpha}) \cap \hb) \subset X(\bm{\alpha}) \oplus L^2(E).\] Now, from \eqref{Jort}, we see that the set $J(X(\bm{\alpha}) \cap \hb) \subset X(\bm{\alpha}) \oplus L^2(E)$ can be expressed as the pre-annihilator.  Then 

\begin{equation} \label{preannih}
    J(X(\bm{\alpha}) \cap \hb) = \{ (bh, \Delta h) \in X(\bm{\alpha^{-1}}) \oplus L^2(E) : h \in H^2 \}^\perp.
\end{equation} 
We are in a Hilbert space setting, so it follows from the duality remarks above that there exists a sequence $(h_n)_n$ of functions in $H^2$, such that $(bh_n, \Delta h_n)$ converges to $(f,g)$ in the norm of $X(\bm{\alpha^{-1}}) \oplus L^2(E)$. By passing to a subsequence, we may assume that the convergence of $\Delta h_n$ to $g$ happens also pointwise almost everywhere on $E$. Mutiplying the second coordinate by the bounded function $b$, we read that the elements $(bh_n, \Delta b h_n)$ converge to $(f,bg)$ in the norm $X(\bm{\alpha^{-1}}) \oplus L^2(E)$, and in particular the tuples $(bh_n, \Delta b h_n)$ form a Cauchy sequence. In fact, since $w = \Delta^2$, this is equivalent to \[ \lim_{n,m \to \infty} \|bh_n - bh_m\|_{X(\bm{\alpha^{-1})}} + \|bh_n - bh_m\|_{L^2(w \, d\m)} = 0, \] so the sequence $(bh_n, bh_n)$ converges in the space $\mathcal{D}(\bm{\alpha}^{-1}, w)$ which is a space of analytic functions. The limit function must be $f \in H^2$, by the above. Now we apply part $(ii)$ of \thref{Dmult}, which tells us that $bh_n \to f$ pointwise on the set $E$ (it is precisely at this point where $M_z$ being completely non-isometric on $D(\bm{\alpha^{-1}}, w)$ is crucial, else the sequence $\{bh_n\}_n$ could potentially converge to something else than $f$). Thus, we have \[ b(\zeta)g(\zeta) = \lim_{n \to \infty} \Delta(\zeta) b(\zeta)h_n(\zeta) = \Delta(\zeta) f(\zeta) \] for almost every $\zeta \in E$. Thus $g = \Delta f/b$ on $E$. All in all, have identified the tuple $Jf$ as \[ Jf = (f, \Delta f/b) \] Moreover, by the permanence property of $(\mathcal{D}(\bm{\alpha}^{-1}, w),\theta) $, $f$ is divisible by the inner factor $\theta$ of $b$. 

Note that on $\T$, we have 
\begin{gather*} 
f/b = (|b|^2 + \Delta^2)f/b = \conj{b}f + \Delta g. 
\end{gather*} Since the right-hand side is square-integrable, so is the left. Since the inner factor of $f$ is divisible by the inner factor of $b$, we conclude that $f/b$ is a function in the Smirnov class of the unit disk which has square-integrable boundary values, and so $f/b \in H^2$ (see \cite[]{garnett}). But by \eqref{hbconteq}, we get \[ f/b = P_+(f/b) = P_+(\conj{b}f + \Delta g) = 0. \] So $f/b \equiv 0$, and hence $f \equiv 0$. The proof is complete.
\end{proof}

\bibliographystyle{siam}
\bibliography{mybib}

\Addresses

\end{document}